\documentclass{amsart}
\usepackage{amssymb}
\usepackage{amsfonts}

\newtheorem{theorem}{Theorem}
\theoremstyle{plain}

\newtheorem{lemma}{Lemma}

\newtheorem{remark}{Remark}

\numberwithin{equation}{section}
\input{tcilatex}

\begin{document}
\title[]{NEW INEQUALITIES OF HERM\.{I}TE-HADAMARD TYPE FOR FUNCTIONS WHOSE
DERIVATIVES ABSOLUTE VALUES ARE QUASI-CONVEX}
\author{\c{C}ET\.{I}N YILDIZ$^{\blacksquare ,\blacklozenge }$}
\address{$^{\blacksquare }$ATAT\"{U}RK UNIVERSITY, K.K. EDUCATION FACULTY,
DEPARTMENT OF MATHEMATICS, 25240, CAMPUS, ERZURUM, TURKEY}
\email{yildizcetiin@yahoo.com}
\author{AHMET OCAK AKDEM\.{I}R$^{\bigstar }$}
\email{merveavci@ymail.com}
\address{$^{\bigstar }$GRADUATE SCHOOL OF NATURAL AND APPLIED SCIENCES, A%
\u{G}RI \.{I}BRAH\.{I}M \c{C}E\c{C}EN UNIVERSITY, A\u{G}RI, TURKEY}
\email{ahmetakdemir@a\u{g}ri.edu.tr}
\thanks{$^{\blacklozenge }$corresponding author.}
\author{MERVE AVCI$^{\blacksquare }$}

\begin{abstract}
In this paper we establish some estimates of the right hand side of a
Hermite-Hadamard type inequality in which some quasi-convex functions are
involved.
\end{abstract}

\maketitle

\section{INTRODUCTION}

Let $f:I\subset 
\mathbb{R}
\rightarrow 
\mathbb{R}
$ be a convex function defined on the interval $I$ of real numbers and a,b$%
\in I$, with $a<b$. The following inequality, known as the \textit{%
Hermite-Hadamard} inequality for convex functions, holds:%
\begin{equation*}
f(\frac{a+b}{2})\leq \frac{1}{b-a}\int_{a}^{b}f(x)dx\leq \frac{f(a)+f(b)}{2}.
\end{equation*}

In recent years many authors have established several inequalities connected
to \textit{Hermite-Hadamard's }inequality. For recent results, refinements,
counterparts, generalizations and new\textit{\ Hermite-Hadamard }type
inequalities see \cite{ssD},\cite{usK} and \cite{gsY}.

We recall that the notion of quasi-convex functions generalizes the notion
of convex functions. More precisely, a function $f:[a,b]\rightarrow 
\mathbb{R}
$ is said to be quasi-convex on $[a,b]$ if%
\begin{equation*}
f(\lambda x+(1-\lambda )y)\leq \max \{f(x),f(y)\},
\end{equation*}%
for any $x,y\in \lbrack a,b]$ and $\lambda \in \lbrack 0,1]$. Clearly, any
convex function is a quasi-convex function. Furthermore, there exist
quasi-convex functions which are not convex (see\cite{daI}).

Recently, D.A. Ion \cite{daI} established two inequalities for functions
whose first derivatives in absolute value are quasi-convex. Namely, he
obtained the following results:

\begin{theorem}
Let $f:I^{o}\subset 
\mathbb{R}
\rightarrow 
\mathbb{R}
$ be a differentiable mapping on $I^{o}$, $a,b\in I^{o}$ with $a<b.$ If $%
\left\vert f^{\prime }\right\vert $ is quasi-convex on $[a,b]$, then the
following inequality holds:%
\begin{equation*}
\left\vert \frac{f(a)+f(b)}{2}-\frac{1}{b-a}\int_{a}^{b}f(u)du\right\vert
\leq \frac{b-a}{4}\left\{ \max \left\vert f^{\prime }(a)\right\vert
,\left\vert f^{\prime }(b)\right\vert \right\} .
\end{equation*}
\end{theorem}

\begin{theorem}
Let $f:I^{o}\subset 
\mathbb{R}
\rightarrow 
\mathbb{R}
$ be a differentiable mapping on $I^{o}$, $a,b\in I^{o}$ with $a<b.$ If $%
\left\vert f^{\prime }\right\vert ^{\frac{p}{p-1}}$ is quasi-convex on $%
[a,b],$ then the following inequality holds:%
\begin{equation*}
\left\vert \frac{f(a)+f(b)}{2}-\frac{1}{b-a}\int_{a}^{b}f(u)du\right\vert
\leq \frac{b-a}{2(p+1)^{\frac{1}{p}}}\left( \max \left\{ \left\vert
f^{\prime }(a)\right\vert ^{\frac{p}{p-1}},\left\vert f^{\prime
}(b)\right\vert ^{\frac{p}{p-1}}\right\} \right) ^{\frac{p-1}{p}}.
\end{equation*}
\end{theorem}

In \cite{mmu}, Alomari et al. obtained the following results.

\begin{theorem}
Let $f:I^{o}\subset \lbrack 0,\infty )\rightarrow 
\mathbb{R}
$ be a differentiable mapping on $I^{o}such$ $that$ $f^{\prime }\in L[a,b],$
where $a,b\in I^{o}$ with $a<b.$ If $\left\vert f^{\prime }\right\vert $ is
quasi-convex on $[a,b]$, then the following inequality holds:%
\begin{eqnarray}
\text{ \ \ \ \ \ \ \ \ \ \ \ }\left\vert \frac{f(a)+f(b)}{2}-\frac{1}{b-a}%
\int_{a}^{b}f(u)du\right\vert &\leq &\frac{b-a}{8}\left[ \max \left\{
\left\vert f^{\prime }\left( \frac{a+b}{2}\right) \right\vert ,\left\vert
f^{\prime }(a)\right\vert \right\} \right.  \label{1} \\
&&\left. +\max \left\{ \left\vert f^{\prime }\left( \frac{a+b}{2}\right)
\right\vert ,\left\vert f^{\prime }(b)\right\vert \right\} \right] .  \notag
\end{eqnarray}
\end{theorem}

\begin{theorem}
Let $f:I^{o}\subset \lbrack 0,\infty )\rightarrow 
\mathbb{R}
$ be a differentiable mapping on $I^{o}such$ $that$ $f^{\prime }\in L[a,b],$
where $a,b\in I^{o}$ with $a<b.$ If $\left\vert f^{\prime }\right\vert ^{%
\frac{p}{p-1}}$ is quasi-convex on $[a,b],$ for $p>1$ then the following
inequality holds:%
\begin{eqnarray}
&&\left\vert \frac{f(a)+f(b)}{2}-\frac{1}{b-a}\int_{a}^{b}f(u)du\right\vert
\label{2} \\
&\leq &\frac{b-a}{4}\left( \frac{1}{p+1}\right) ^{\frac{1}{p}}\left[ \left(
\max \left\{ \left\vert f^{\prime }(\frac{a+b}{2})\right\vert ^{\frac{p}{p-1}%
},\left\vert f^{\prime }(a)\right\vert ^{\frac{p}{p-1}}\right\} \right) ^{%
\frac{p-1}{p}}\right.  \notag \\
&&\left. +\left( \max \left\{ \left\vert f^{\prime }(\frac{a+b}{2}%
)\right\vert ^{\frac{p}{p-1}},\left\vert f^{\prime }(b)\right\vert ^{\frac{p%
}{p-1}}\right\} \right) ^{\frac{p-1}{p}}\right] .  \notag
\end{eqnarray}
\end{theorem}

\begin{theorem}
Let $f:I^{o}\subset 
\mathbb{R}
\rightarrow 
\mathbb{R}
$ be a differentiable mapping on $I^{o}$, $a,b\in I^{o}$ with $a<b.$ If $%
\left\vert f^{\prime }\right\vert ^{q}$ is quasi-convex on $[a,b],q\geq 1,$
then the following inequality holds:%
\begin{eqnarray}
&&\left\vert \frac{f(a)+f(b)}{2}-\frac{1}{b-a}\int_{a}^{b}f(u)du\right\vert
\label{3} \\
&\leq &\frac{b-a}{8}\left[ \left( \max \left\{ \left\vert f^{\prime }(\frac{%
a+b}{2})\right\vert ^{q},\left\vert f^{\prime }(a)\right\vert ^{q}\right\}
\right) ^{\frac{1}{q}}\right.  \notag \\
&&\left. +\left( \max \left\{ \left\vert f^{\prime }(\frac{a+b}{2}%
)\right\vert ^{q},\left\vert f^{\prime }(b)\right\vert ^{q}\right\} \right)
^{\frac{1}{q}}\right] .  \notag
\end{eqnarray}
\end{theorem}

The main purpose of this study is to generalize the Theorem 3, Theorem 4 and
Theorem 5 for quasi-convex functions using the new Lemma.

\section{HERM\.{I}TE-HADAMARD TYPE INEQUALITIES}

\begin{lemma}
Let $f:I\subset 
\mathbb{R}
\rightarrow 
\mathbb{R}
$ be a differentiable mapping on $I^{o}$ where $a,b\in I$ with $a<b.$ If $%
f^{\prime }\in L[a,b],$ then the following equality holds:%
\begin{eqnarray*}
\frac{(b-x)f(b)+(x-a)f(a)}{b-a}-\frac{1}{b-a}\int_{a}^{b}f(u)du &=&\frac{%
(x-a)^{2}}{b-a}\int_{0}^{1}(t-1)f^{\prime }(tx+(1-t)a)dt \\
&&+\frac{(b-x)^{2}}{b-a}\int_{0}^{1}(1-t)f^{\prime }(tx+(1-t)b)dt.
\end{eqnarray*}
\end{lemma}

\begin{theorem}
Let $f:I^{o}\subset 
\mathbb{R}
\rightarrow 
\mathbb{R}
$ be a differentiable mapping on $I^{o}$, $a,b\in I^{o}$ with $a<b.$ If $%
\left\vert f^{\prime }\right\vert $ is quasi-convex on $[a,b]$, then the
following inequality holds:%
\begin{eqnarray*}
\left\vert \frac{(b-x)f(b)+(x-a)f(a)}{b-a}-\frac{1}{b-a}\int_{a}^{b}f(u)du%
\right\vert &\leq &\frac{(x-a)^{2}}{2(b-a)}\max \left\{ \left\vert f^{\prime
}(x)\right\vert ,\left\vert f^{\prime }(a)\right\vert \right\} \\
&&+\frac{(b-x)^{2}}{2(b-a)}\max \{\left\vert f^{\prime }(x)\right\vert
,\left\vert f^{\prime }(b)\right\vert \}.
\end{eqnarray*}
\end{theorem}

\begin{proof}
From Lemma 1, we have%
\begin{eqnarray*}
&&\left\vert \frac{(b-x)f(b)+(x-a)f(a)}{b-a}-\frac{1}{b-a}%
\int_{a}^{b}f(u)du\right\vert \\
&\leq &\frac{(x-a)^{2}}{b-a}\int_{0}^{1}(1-t)\left\vert f^{\prime
}(tx+(1-t)a)\right\vert dt \\
&&+\frac{(b-x)^{2}}{b-a}\int_{0}^{1}(1-t)\left\vert f^{\prime
}(tx+(1-t)b)\right\vert dt \\
&\leq &\frac{(x-a)^{2}}{b-a}\int_{0}^{1}(1-t)\max \{\left\vert f^{\prime
}(x)\right\vert ,\left\vert f^{\prime }(a)\right\vert \}dt \\
&&+\frac{(b-x)^{2}}{b-a}\int_{0}^{1}(1-t)\max \{\left\vert f^{\prime
}(x)\right\vert ,\left\vert f^{\prime }(b)\right\vert \}dt \\
&=&\frac{(x-a)^{2}}{b-a}\max \{\left\vert f^{\prime }(x)\right\vert
,\left\vert f^{\prime }(a)\right\vert \}\int_{0}^{1}(1-t)dt \\
&&+\frac{(b-x)^{2}}{b-a}\max \{\left\vert f^{\prime }(x)\right\vert
,\left\vert f^{\prime }(b)\right\vert \}\int_{0}^{1}(1-t)dt \\
&=&\frac{(x-a)^{2}}{2(b-a)}\max \{\left\vert f^{\prime }(x)\right\vert
,\left\vert f^{\prime }(a)\right\vert \} \\
&&+\frac{(b-x)^{2}}{2(b-a)}\max \{\left\vert f^{\prime }(x)\right\vert
,\left\vert f^{\prime }(b)\right\vert \},
\end{eqnarray*}%
which completes the proof.
\end{proof}

\begin{remark}
In Theorem 6,if we choose $x=\frac{a+b}{2},$ we obtain (\ref{1}) inequality.
\end{remark}

\begin{theorem}
Let $f:I^{o}\subset 
\mathbb{R}
\rightarrow 
\mathbb{R}
$ be a differentiable mapping on $I^{o}$, $a,b\in I^{o}$ with $a<b.$ If $%
\left\vert f^{\prime }\right\vert ^{\frac{p}{p-1}}$ is quasi-convex on $%
[a,b],p>1,$ then the following inequality holds:%
\begin{eqnarray*}
&&\left\vert \frac{(b-x)f(b)+(x-a)f(a)}{b-a}-\frac{1}{b-a}%
\int_{a}^{b}f(u)du\right\vert \\
&\leq &\frac{(x-a)^{2}}{b-a}\left( \frac{1}{p+1}\right) ^{\frac{1}{p}}\left(
\max \left\{ \left\vert f^{\prime }(x)\right\vert ^{\frac{p}{p-1}%
},\left\vert f^{\prime }(a)\right\vert ^{\frac{p}{p-1}}\right\} \right) ^{%
\frac{p-1}{p}} \\
&&+\frac{(b-x)^{2}}{b-a}\left( \frac{1}{p+1}\right) ^{\frac{1}{p}}\left(
\max \left\{ \left\vert f^{\prime }(x)\right\vert ^{\frac{p}{p-1}%
},\left\vert f^{\prime }(b)\right\vert ^{\frac{p}{p-1}}\right\} \right) ^{%
\frac{p-1}{p}}
\end{eqnarray*}%
where $q=p/(p-1).$
\end{theorem}

\begin{proof}
From Lemma 1 and using well known H\"{o}lder inequality, we have 
\begin{eqnarray*}
&&\left\vert \frac{(b-x)f(b)+(x-a)f(a)}{b-a}-\frac{1}{b-a}%
\int_{a}^{b}f(u)du\right\vert \\
&\leq &\frac{(x-a)^{2}}{b-a}\int_{0}^{1}(1-t)\left\vert f^{\prime
}(tx+(1-t)a)\right\vert dt \\
&&+\frac{(b-x)^{2}}{b-a}\int_{0}^{1}(1-t)\left\vert f^{\prime
}(tx+(1-t)b)\right\vert dt \\
&\leq &\frac{(x-a)^{2}}{b-a}\left( \int_{0}^{1}(1-t)^{p}dt\right) ^{\frac{1}{%
p}}\left( \int_{0}^{1}\left\vert f^{\prime }(tx+(1-t)a)\right\vert ^{\frac{p%
}{p-1}}dt\right) ^{\frac{p-1}{p}} \\
&&+\frac{(b-x)^{2}}{b-a}\left( \int_{0}^{1}(1-t)^{p}dt\right) ^{\frac{1}{p}%
}\left( \int_{0}^{1}\left\vert f^{\prime }(tx+(1-t)b)\right\vert ^{\frac{p}{%
p-1}}dt\right) ^{\frac{p-1}{p}} \\
&\leq &\frac{(x-a)^{2}}{b-a}\left( \int_{0}^{1}(1-t)^{p}dt\right) ^{\frac{1}{%
p}}\left( \int_{0}^{1}\max \left\{ \left\vert f^{\prime }(x)\right\vert ^{%
\frac{p}{p-1}},\left\vert f^{\prime }(a)\right\vert ^{\frac{p}{p-1}}\right\}
dt\right) ^{\frac{p-1}{p}} \\
&&+\frac{(b-x)^{2}}{b-a}\left( \int_{0}^{1}(1-t)^{p}dt\right) ^{\frac{1}{p}%
}\left( \int_{0}^{1}\max \left\{ \left\vert f^{\prime }(x)\right\vert ^{%
\frac{p}{p-1}},\left\vert f^{\prime }(b)\right\vert ^{\frac{p}{p-1}}\right\}
dt\right) ^{\frac{p-1}{p}} \\
&=&\frac{(x-a)^{2}}{b-a}\left( \frac{1}{p+1}\right) ^{\frac{1}{p}}\left(
\max \left\{ \left\vert f^{\prime }(x)\right\vert ^{\frac{p}{p-1}%
},\left\vert f^{\prime }(a)\right\vert ^{\frac{p}{p-1}}\right\} \right) ^{%
\frac{p-1}{p}} \\
&&+\frac{(b-x)^{2}}{b-a}\left( \frac{1}{p+1}\right) ^{\frac{1}{p}}\left(
\max \left\{ \left\vert f^{\prime }(x)\right\vert ^{\frac{p}{p-1}%
},\left\vert f^{\prime }(b)\right\vert ^{\frac{p}{p-1}}\right\} \right) ^{%
\frac{p-1}{p}}
\end{eqnarray*}%
where $\frac{1}{p}+\frac{1}{q}=1,$ which completes the proof.
\end{proof}

\begin{remark}
In Theorem 7, if we choose $x=\frac{a+b}{2},$ we obtain (\ref{2}) inequality.
\end{remark}

\begin{theorem}
Let $f:I^{o}\subset 
\mathbb{R}
\rightarrow 
\mathbb{R}
$ be a differentiable mapping on $I^{o}$, $a,b\in I^{o}$ with $a<b.$ If $%
\left\vert f^{\prime }\right\vert ^{q}$ is quasi-convex on $[a,b],q\geq 1,$
then the following inequality holds:%
\begin{eqnarray*}
\left\vert \frac{(b-x)f(b)+(x-a)f(a)}{b-a}-\frac{1}{b-a}\int_{a}^{b}f(u)du%
\right\vert &\leq &\frac{(x-a)^{2}}{2\left( b-a\right) }\left( \max \left\{
\left\vert f^{\prime }(x)\right\vert ^{q},\left\vert f^{\prime
}(a)\right\vert ^{q}\right\} \right) ^{\frac{1}{q}} \\
&&+\frac{(b-x)^{2}}{2\left( b-a\right) }\left( \max \left\{ \left\vert
f^{\prime }(x)\right\vert ^{q},\left\vert f^{\prime }(b)\right\vert
^{q}\right\} \right) ^{\frac{1}{q}}.
\end{eqnarray*}
\end{theorem}

\begin{proof}
From Lemma 1 and using the well known power mean inequality, we have%
\begin{eqnarray*}
&&\left\vert \frac{(b-x)f(b)+(x-a)f(a)}{b-a}-\frac{1}{b-a}%
\int_{a}^{b}f(u)du\right\vert \\
&\leq &\frac{(x-a)^{2}}{b-a}\int_{0}^{1}(1-t)\left\vert f^{\prime
}(tx+(1-t)a)\right\vert dt \\
&&+\frac{(b-x)^{2}}{b-a}\int_{0}^{1}(1-t)\left\vert f^{\prime
}(tx+(1-t)b)\right\vert dt \\
&\leq &\frac{(x-a)^{2}}{b-a}\left( \int_{0}^{1}(1-t)dt\right) ^{1-\frac{1}{q}%
}\left( \int_{0}^{1}(1-t)\left\vert f^{\prime }(tx+(1-t)a)\right\vert
^{q}dt\right) ^{\frac{1}{q}} \\
&&+\frac{(b-x)^{2}}{b-a}\left( \int_{0}^{1}(1-t)dt\right) ^{1-\frac{1}{q}%
}\left( \int_{0}^{1}(1-t)\left\vert f^{\prime }(tx+(1-t)b)\right\vert
^{q}dt\right) ^{\frac{1}{q}}.
\end{eqnarray*}%
Since $\left\vert f^{\prime }\right\vert ^{q}$ is quasi-convex we have%
\begin{equation*}
\int_{0}^{1}(1-t)\left\vert f^{\prime }(tx+(1-t)a)\right\vert ^{q}dt\leq 
\frac{1}{2}\max \{\left\vert f^{\prime }(x)\right\vert ^{q},\left\vert
f^{\prime }(a)\right\vert ^{q}\}
\end{equation*}%
and%
\begin{equation*}
\int_{0}^{1}(1-t)\left\vert f^{\prime }(tx+(1-t)b)\right\vert ^{q}dt\leq 
\frac{1}{2}\max \{\left\vert f^{\prime }(x)\right\vert ^{q},\left\vert
f^{\prime }(b)\right\vert ^{q}\}.
\end{equation*}%
Therefore, we have%
\begin{eqnarray*}
\left\vert \frac{(b-x)f(b)+(x-a)f(a)}{b-a}-\frac{1}{b-a}\int_{a}^{b}f(u)du%
\right\vert &\leq &\frac{(x-a)^{2}}{2\left( b-a\right) }\left( \max \left\{
\left\vert f^{\prime }(x)\right\vert ^{q},\left\vert f^{\prime
}(a)\right\vert ^{q}\right\} \right) ^{\frac{1}{q}} \\
&&+\frac{(b-x)^{2}}{2\left( b-a\right) }\left( \max \left\{ \left\vert
f^{\prime }(x)\right\vert ^{q},\left\vert f^{\prime }(b)\right\vert
^{q}\right\} \right) ^{\frac{1}{q}}.
\end{eqnarray*}
\end{proof}

\begin{remark}
In Theorem 8, if we choose $x=\frac{a+b}{2},$ we obtain (\ref{3}) inequality.
\end{remark}

\end{document}